\documentclass{proc-l}

\usepackage{amssymb}

\newcommand{\na}{\mathbb{N}}
\newcommand{\re}{\mathbb{R}}

\newcommand{\rn}{\mathbb{R}^n}

\newcommand{\norm}[2]{\|#1\|_{#2}}

\newcommand{\calE}{\mathcal{E}}
\newcommand{\calF}{\mathcal{F}}

\newtheorem{theorem}{Theorem}[section]

\theoremstyle{definition}
\newtheorem{definition}[theorem]{Definition}

\theoremstyle{remark}
\newtheorem{remark}[theorem]{Remark}

\numberwithin{equation}{section}
\usepackage{enumerate}

\begin{document}

\title[From Sobolev Inequality to Doubling]{From Sobolev Inequality to Doubling}


\author[Korobenko]{Lyudmila Korobenko}
\address{University of Calgary\\
Calgary, Alberta\\
lkoroben@ucalgary.ca}
\author[Maldonado]{Diego Maldonado}
\address{Kansas State University\\
Manhattan, Kansas\\
dmaldona@math.ksu.edu}
\author[Rios]{Cristian Rios}
\address{University of Calgary\\
Calgary, Alberta\\
crios@ucalgary.ca}
\thanks{Second author supported by the US National Science Foundation under grant DMS 1361754. Third author supported by the Natural Sciences and
Engineering Research Council of Canada.}


\date{}

\dedicatory{}

\commby{Jeremy Tyson}

\begin{abstract} In various analytical contexts, it is proved that a weak Sobolev inequality implies a doubling property for the underlying measure.

\end{abstract}

\maketitle

\section{Introduction and main result}\label{sec:intro}

Let $d_E$ denote the usual Euclidean distance in $\rn$, that is, $d_E(x,y):=|x-y|$ for every $x, y \in \rn$. Given $y \in \rn$ and $R > 0$ let $B(y,R):=\{x \in \rn : d_E(y,x) < R\}$ denote the Euclidean ball centered at $y$ with radius $R$.

\begin{definition}
Let  $\mu$ be a Borel measure on $(\rn, d_E)$ such that $0 < \mu(B) < \infty$ for every Euclidean ball $B \subset \rn$.  Given $1\leq p < \infty$ and $1< \sigma <\infty$, we say that the triple $(\rn, d_E, \mu)$ admits a \emph{weak $(p\sigma, p)$-Sobolev inequality} with a (finite) constant $C_S > 0$ if for every Euclidean ball $B:=B(y,R) \subset \rn$ and every function $\varphi \in C_c^{0,1}(B)$ (Lipschitz and compactly supported on $B$) it holds true that
\begin{align}\label{weakEuclideanSob}
\left(\frac{1}{\mu(B)} \int_B |\varphi|^{p \sigma} d\mu \right)^{\frac{1}{p \sigma}} & \leq C_S R \left(\frac{1}{\mu(B)} \int_B |\nabla \varphi|^p d\mu \right)^{\frac{1}{p}}\\\nonumber
& + C_S \left(\frac{1}{\mu(B)} \int_B |\varphi|^p d\mu\right)^{\frac{1}{p}}.
\end{align}
\end{definition}
Our main result is

\begin{theorem}\label{thm:Sob=>doubling} Suppose that, for some $1\leq p < \infty$ and $1< \sigma <\infty$, the triple $(\rn, d_E, \mu)$ admits a weak $(p\sigma, p)$-Sobolev inequality with a constant $C_S >0$. Then, the measure $\mu$ is doubling on $(\rn, d_E)$. More precisely, there exists a constant $C_D \geq 1$, depending only on $p$, $\sigma$, and $C_S$, such that
\begin{equation}\label{mudoubling}
\mu(B(y,2R)) \leq C_D \, \mu(B(y,R)) \quad \forall y \in \rn, R > 0.
\end{equation}
\end{theorem}

\begin{remark}  In \cite{FKS82},  Fabes, Kenig and Serapioni identified four conditions on an absolutely continuous measure $d\mu = w \,dx$ on $(\mathbb{R}^n, d_E)$ as essential in proving Harnack's inequality for solutions to certain degenerate elliptic PDEs (whose degeneracy is ruled by $w$) by means of the implementation of Moser's iterative scheme and John-Nirenberg-type inequalities. For $1 \leq p < \infty$, these conditions, which we next list for the reader's convenience, define $w$ as a $p$-\emph{admissible weight} (see \cite[Section 13]{Haj2} and  \cite[p.7]{HKM}):
\begin{enumerate}[(I)]
\item the doubling property: there exists $C > 0$ such that $\mu(2B) \leq C \mu(B)$ for every Euclidean ball $B \subset \mathbb{R}^n$;\label{list-doubling}
\item the uniqueness condition for the gradient: if $D \subset \mathbb{R}^n$ is an open set and $\{\varphi_j\}_j \subset C^\infty(D)$ satisfy $\int_D |\varphi_j|^p d\mu \rightarrow 0$ and $\int_D |\nabla \varphi_ j - v|^p d\mu \rightarrow 0$ as $j \rightarrow \infty$ for some $v \in L^p(D, \mu)$, then $v \equiv 0$;\label{list-uniq}
\item the $p$-Sobolev inequality: there exist $\sigma > 1$ and $C > 0$ such that for every ball $B:=B(y,R)$ and $\varphi \in C^\infty_c(B)$ it holds \label{list-sobolev}
\begin{equation}\label{p-Sobolev}
\left(\frac{1}{\mu(B)} \int_B |\varphi|^{\sigma p} \, d\mu \right)^{\frac{1}{\sigma p}} \leq C R \left(\frac{1}{\mu(B)} \int_B |\nabla \varphi|^p \, d\mu \right)^{\frac{1}{p}};
\end{equation}
\item the $p$-Poincar\'e inequality: there exists $C > 0$ such that for every ball $B:=B(y,R)$ and $\varphi \in C^\infty(B)$ it holds \label{list-poincare}
\[
\int_B |\varphi - \varphi_B|^p \, d\mu \leq C R^p \int_B |\nabla \varphi|^p \, d\mu,
\]
where $\varphi_B$ stands for the average of $\varphi$ over $B$ with respect to $d\mu$.
\end{enumerate}
The interplay between the conditions above has received considerable attention. Independently,  Saloff-Coste \cite{SC1} and Grygor'yan \cite{Gr} proved that \eqref{list-doubling} and \eqref{list-poincare} imply \eqref{list-sobolev} (and this implication was later systematized by several authors, see \cite[p.79]{Haj2}). Then, Heinonen and Koskela  \cite[Theorem 5.2]{HK95} proved that, as it had been announced by S. Semmes, conditions \eqref{list-doubling} and \eqref{list-poincare} imply \eqref{list-uniq}.

Our Theorem \ref{thm:Sob=>doubling} then contributes to the further understanding of such interplay by establishing that a weaker version of \eqref{list-sobolev} (as given by \eqref{weakEuclideanSob}) implies \eqref{list-doubling}, thus placing the Sobolev inequality \eqref{p-Sobolev} back into the core of the regularity theory aspects of partial differential equations.
\end{remark}

\subsection{Metric spaces} The theory of Sobolev-type inequalities in metric spaces, including imbedding theorems and several definitions of Sobolev spaces, has seen great developments in the past two decades, see for instance \cite{BB11, Fr1, Fr2, Haj,Haj2, Shan} and references therein. A driving motivation for the study of such Sobolev-type inequalities arises in the study of regularity of solutions to certain classes of degenerate elliptic and parabolic PDEs, see for instance \cite[Chapters 7-14]{BB11}, \cite{BM, DJS13, KS, MM} and references therein. In what follows, we briefly review the theory of Sobolev spaces in metric spaces as to formulate a corresponding version of Theorem \ref{thm:Sob=>doubling}. The reader is referred to \cite[Chapters 1-5]{BB11}, \cite[Sections 1-5]{BM}, \cite[Sections 1-3]{KS} for further details.

Let $(X,d)$ be a metric space. For $y \in X$ and $R> 0$ the $d$-ball centered at $y$ with radius $R$ is defined as $B_d(y,R):=\{x \in X: d(x,y) < R\}$.

Given a function $u: X \rightarrow [-\infty, \infty]$, a non-negative Borel function $g: X\rightarrow [0,\infty]$ is called an \emph{upper gradient} of $u$ if for all curves (i.e. non-constant rectifiable continuous mappings) $\gamma : [0, l_\gamma] \rightarrow X$ it holds
$$
|u(\gamma(0)) - u(\gamma (l_\gamma))| \leq \int_\gamma g \, ds.
$$
In particular, if for some $L \geq 1$, $u: X \rightarrow \re$ is an \emph{$L$-Lipschitz function}, i.e., $|u(x)-u(y)| \leq L \, d(x,y)$ for every $x,y\in X$, then the function $lip(u)$ defined for $x \in X$ as
\begin{equation}\label{deflipu}
lip(u)(x):= \liminf_{r \to 0^+} \sup\limits_{y \in B_d(x,r)} \frac{|u(x)-u(y)|}{r}
\end{equation}
is an upper-gradient for $u$ (see \cite[Proposition 1.14]{BB11}). Notice also that $lip(u)(x) \leq L$ for every $x \in X$.

Let $\mu$ be a Borel measure on $(X,d)$ such that $0 < \mu(B) < \infty$ for every $d$-ball $B \subset X$. For $1 \leq p < \infty$ and $u \in L^p(X, \mu)$ set
$$
\norm{u}{N^{1,p}}^p:= \int_X |u|^p d\mu + \inf\limits_{g} \int_X g^p d\mu,
$$
where the infimum is taken over all the upper-gradients $g$ of $u$. Given a $d$-ball $B \subset X$, its Newtonian space with zero boundary values is defined as
$$
N_0^{1,p}(B):=\{f|_B : \norm{f}{N^{1,p}} < \infty \text{ and } f\equiv 0 \text{ on } X \setminus B\}.
$$

\begin{definition}
Let $(X, d, \mu)$ be as above. Given $1\leq p < \infty$ and $1< \sigma <\infty$, we say that the triple $(X, d, \mu)$ admits a \emph{weak $(p\sigma, p)$-Sobolev inequality} with a (finite) constant $C_S > 0$ if for every $d$-ball $B:=B_d(y,R) \subset X$ and every function $\varphi \in N_0^{1,p}(B)$ it holds true that
\begin{align}\label{weakMetricSob}
\left(\frac{1}{\mu(B)} \int_B |\varphi|^{p \sigma} d\mu \right)^{\frac{1}{p \sigma}} & \leq C_S R \left(\frac{1}{\mu(B)} \int_B g^p d\mu \right)^{\frac{1}{p}}\\\nonumber
& + C_S \left(\frac{1}{\mu(B)} \int_B |\varphi|^p d\mu\right)^{\frac{1}{p}},
\end{align}
for all upper-gradients $g$ of $\varphi$.
\end{definition}

Then we have the following metric-space formulation of Theorem \ref{thm:Sob=>doubling}

\begin{theorem}\label{thm:Sob=>doublingMetric} Suppose that, for some $1\leq p < \infty$ and $1< \sigma <\infty$, the triple $(X, d, \mu)$ admits a weak $(p\sigma, p)$-Sobolev inequality with a constant $C_S >0$. Then, the measure $\mu$ is doubling on $(X, d)$. More precisely, there exists a constant $C_D \geq 1$, depending only on $p$, $\sigma$, and $C_S$, such that
\begin{equation*}
\mu(B_d(y,2R)) \leq C_D \, \mu(B_d(y,R)) \quad \forall y \in X, R >0.
\end{equation*}
\end{theorem}

Clearly, Theorem \ref{thm:Sob=>doubling} then follows as a corollary of Theorem \ref{thm:Sob=>doublingMetric}.

\section{Proof of Theorem \ref{thm:Sob=>doublingMetric}}\label{sec:proofThm}

Given a $d$-ball $B:=B_d(y,R)$ set $B^*:=B_d(y,2R)$ and define a family of $d$-Lipschitz functions $\{\psi_j\}_{j \in \na} \subset N_0^{1,p}(B) \subset N_0^{1,p}(B^*)$ as follows: for $j \in \na$ set $r_j:= (2^{-j-1} + 2^{-1}) R$ and
\begin{equation}\label{defpsijMetric}
\psi_j(x):= \left(\frac{r_j - d(x,y)}{r_j - r_{j+1}} \right)^+ \wedge 1.
\end{equation}
Also for $j \in \na$, define the $d$-balls $B_j$ as
$$
\frac{1}{2}B \subset B_j:= \{x \in X: d(x,y) \leq r_j\} \subset B \subset B^*=2B.
$$
Our first step will be to apply the weak-Sobolev inequality \eqref{weakMetricSob} to $\psi_j$ on $B^*$ by choosing the upper-gradient $g_j:=lip(\psi_j)$ as defined in \eqref{deflipu}. In particular, it follows that
\begin{equation}\label{condpsij}
g_j(x) \leq \frac{2^{j+2}}{R} \chi_{B_j}(x) \quad \text{and} \quad 0 \leq \psi_j(x) \leq 1 \quad \forall x \in X.
\end{equation}
Then, by the weak-Sobolev inequality \eqref{weakMetricSob}  applied to each $\psi_j$ on $B^*=2B$ (and using the fact that each $\psi_j$ is supported in $B_j$, so that $B^*$ can be replaced by $B_j$ in the integrals), we obtain
\begin{align}\label{Sobappliedtopsij}
\left(\frac{1}{\mu(B^*)} \int_{B_j} |\psi_j|^{p \sigma} d\mu \right)^{\frac{1}{p \sigma}} & \leq 2 C_S R \left(\frac{1}{\mu(B^*)} \int_{B_j} g_j^p d\mu \right)^{\frac{1}{p}}\\\nonumber
& + C_S \left(\frac{1}{\mu(B^*)} \int_{B_j} |\psi_j|^p d\mu\right)^{\frac{1}{p}}.
\end{align}
By the estimates for $g_j$ and $\psi_j$ in \eqref{condpsij}, we obtain
\begin{equation}\label{boundGradPsij}
\left(\frac{1}{\mu(B^*)} \int_{B_j} g_j^p d\mu \right)^{\frac{1}{p}} \leq \frac{2^{j+2}}{R} \left(\frac{\mu(B_j)}{\mu(B^*)} \right)^{\frac{1}{p}}
\end{equation}
and
\begin{equation}\label{boundPsijp}
\left(\frac{1}{\mu(B^*)} \int_{B_j} |\psi_j|^p d\mu\right)^{\frac{1}{p}} \leq  \left(\frac{\mu(B_j)}{\mu(B^*)} \right)^{\frac{1}{p}}.
\end{equation}
On the other hand, since $\psi_j \equiv 1$ on $B_{j+1} \subset B_j$, it follows that
\begin{equation}\label{boundEjplus1}
 \left(\frac{\mu(B_{j+1})}{\mu(B^*)} \right)^{\frac{1}{p \sigma}} \leq \left(\frac{1}{\mu(B^*)} \int_{B_j} |\psi_j|^{p \sigma} d\mu \right)^{\frac{1}{p \sigma}}.
\end{equation}
Therefore, by combining \eqref{boundEjplus1}, \eqref{boundPsijp}, and \eqref{boundGradPsij} with \eqref{Sobappliedtopsij}, we obtain
\begin{align*}
 \left(\frac{\mu(B_{j+1})}{\mu(B^*)} \right)^{\frac{1}{p \sigma}} & \leq 2 C_S 2^{j+2}  \left(\frac{\mu(B_j)}{\mu(B^*)} \right)^{\frac{1}{p}} + C_S  \left(\frac{\mu(B_j)}{\mu(B^*)} \right)^{\frac{1}{p}} \leq C_S 2^{j+4} \left(\frac{\mu(B_j)}{\mu(B^*)} \right)^{\frac{1}{p}} .
\end{align*}
Raising to the power $p/\sigma^{j-1}$ then yields
\begin{equation}\label{preiterationPj}
 \left(\frac{\mu(B_{j+1})}{\mu(B^*)} \right)^{\frac{1}{\sigma^j}} \leq C_S^{\frac{p}{\sigma^{j-1}}} 2^{\frac{p(j+4)}{\sigma^{j-1}}}  \left(\frac{\mu(B_j)}{\mu(B^*)} \right)^{\frac{1}{\sigma^{j-1}}}.
\end{equation}
At this point, for $j \in \na$, define $P_j:= \mu(B_j)^{\frac{1}{\sigma^{j-1}}}$, so that \eqref{preiterationPj} can be recast as
\begin{equation}\label{preiterationPj2}
P_{j+1} \leq C_S^{\frac{\sigma p}{\sigma^{j}}} 2^{\frac{\sigma p (j+4)}{\sigma^{j}}} \mu(B^*)^{\frac{1-\sigma}{\sigma^j}} P_j.
\end{equation}
Notice that, from the construction of $B_j$, we have $B_d(y,R/2) \subset B_j \subset B$ for every $j \in \na$, so that $0 < \mu(B_d(y,R/2)) \leq \mu(B_j) \leq \mu(B) < \infty$ for every $j \in \na$. Then
$$
\mu(B_d(y,R/2))^{\frac{1}{\sigma^{j-1}}} \leq P_j \leq \mu(B)^{\frac{1}{\sigma^{j-1}}} \quad \forall j \in \na,
$$
which implies $\lim\limits_{j\to\infty}P_j =1$. Now, by iterations of \eqref{preiterationPj2}, we get
\begin{equation}\label{postiterationsPsij}
1=\lim\limits_{j\to\infty}P_j \leq P_1 \prod\limits_{j=1}^\infty [C_S^{\sigma p} 2^{\sigma p (j+4)} \mu(B^*)^{1-\sigma}]^{\frac{1}{\sigma^j}},
\end{equation}
with
$$
\prod\limits_{j=1}^\infty [C_S^{\sigma p}]^{\frac{1}{\sigma^j}} = \exp \left[\sigma p \left( \sum\limits_{j=1}^\infty \frac{1}{\sigma^j} \right)  \log C_S \right] = C_S^{\frac{p \sigma}{\sigma -1}},
$$
$$
\prod\limits_{j=1}^\infty [2^{p \sigma (j+4)}]^{\frac{1}{\sigma^j}} = \exp \left[\sigma p \left( \sum\limits_{j=1}^\infty \frac{j+4}{\sigma^j} \right)  \log 2 \right] =: K_1(\sigma,p)<\infty,
$$
and
\begin{align*}
\prod\limits_{j=1}^\infty [\mu(B^*)^{1-\sigma}]^{\frac{1}{\sigma^j}} & = \exp \left[ \left( \sum\limits_{j=1}^\infty \frac{(1-\sigma)}{\sigma^j} \right)  \log \mu(B^*) \right]\\
& = \exp \left[ \sum\limits_{j=1}^\infty \left( \frac{1}{\sigma^j}- \frac{1}{\sigma^{j-1}} \right)  \log \mu(B^*) \right] = \frac{1}{\mu(B^*)}.
\end{align*}
Consequently, \eqref{postiterationsPsij} yields
$$
1 \leq  C_S^{\frac{p \sigma}{\sigma -1}} K_1(\sigma, p) \frac{P_1}{\mu(B^*)},
$$
which, together with the fact that $P_1 = \mu(B_1) \leq \mu(B)$, implies
$$
\mu(B^*) \leq C_S^{\frac{p \sigma}{\sigma -1}} K_1(\sigma, p) \mu(B)
$$
and \eqref{mudoubling} follows with $C_D:= C_S^{\frac{p \sigma}{\sigma -1}} K_1(\sigma, p)$.
\qed

\section{Further extensions of Theorem \ref{thm:Sob=>doubling}}

From the proof of Theorem \ref{thm:Sob=>doublingMetric} the central role of the functions $\{\psi_j\}$ becomes quite apparent. In this section we reformulate Theorem \ref{thm:Sob=>doubling} in other contexts where corresponding functions $\{\psi_j\}$ can be constructed.

\subsection{Dirichlet forms} Since the 1990's, deep connections between Sobolev and Poincar\'e inequalities, doubling properties for measures, and elliptic and parabolic Harnack inequalities have been discovered and further developed in the ample context of strongly local, regular Dirichlet forms (including analysis on complete Riemannian manifolds, Alexandrov spaces, self-similar sets, graphs, etc.) see, for instance, \cite{BGK, BM1, BM2, BM3, Gr,   Ki, KSW,KST, Ku, Kuw,  KMS, SC1, SC2, St1, St2, St3, St4} and references therein. In this subsection we will recast Theorem \ref{thm:Sob=>doubling} in the language of Dirichlet forms. A brief review of some basic notions is in order.

Let $(X,\tau)$ be a Hausdorff, locally compact, separable topological space and let $\mu$ be a Radon measure on $(X,\tau)$ such that $\mu(U) >0$ for every nonempty open subset $U \subset X$.

Let $\calF$ be a dense subspace of $L^2(X,\mu):=\{u:X \rightarrow \re : \int_X u^2 d\mu < \infty \}$ and let $\calE : \calF \times \calF \rightarrow [0,\infty)$ be a bilinear, non-negative definite (that is, $\calE(u,u) \geq 0 \: \forall u \in \calF$), and symmetric functional. For every $u \in \calF$ set $\calE(u):=\calE(u,u)$. Assume that $(\calE, \calF)$ is \emph{closed}, that is, $\calF$ equipped with the norm $\norm{u}{\calF}:=(\norm{u}{L^2(X,\mu)} + \calE(u))^{1/2}$ becomes a Hilbert space and that $(\calE, \calF)$ is \emph{Markovian}; that is, for every $u \in \calF$, it follows that $u_1:= (0 \vee u ) \wedge 1 \in \calF$ and $\calE(u_1) \leq \calE(u)$. When all the above conditions are met, $(\calF, \calE)$ is called a \emph{Dirichlet form} on $L^2(X,\mu)$. We refer the reader to \cite[Chapters 1-3]{FOT} for further details and properties of Dirichlet forms.

We intend to state a Sobolev-type inequality along the lines of \eqref{weakEuclideanSob} involving a Dirichlet form $(\calF, \calE)$; hence, the next step will be about imposing conditions on $(\calF, \calE)$ as to equip $X$ (which thus far is just a topological, and not necessarily metric, space) with a convenient distance.

Let $C_c(X)$ denote the class of real-valued, continuous functions on $X$ with compact support equipped with the uniform topology. Following the notation in \cite[Section 1.1]{FOT}, a Dirichlet form is \emph{regular} if $\calF \cap C_c(X)$ is dense in both $(\calF, \norm{\cdot}{\calF})$ and $(C_c(X), \norm{\cdot}{L^\infty(X)})$ and it is \emph{strongly local} if $\calE(u,v)=0$ for all $u,v \in \calF$ with $u \equiv 1$ on a neighborhood of $\text{supp}(v)$. A strongly local regular Dirichlet form $(\calF, \calE)$ admits the integral representation
\begin{equation}
\calE(u,v) = \int_X d\Gamma(u,v)  \quad \forall u,v \in \calF,
\end{equation}
where $\Gamma$ (called the \emph{energy measure} of $(\calF, \calE)$) is a bilinear, non-negative definite, symmetric form with values in the signed Radon measures of $X$ (see \cite[Section 3.2]{FOT}).  Moreover, the energy measure has a local character, meaning that given $u, v \in \calF$ and an open set $\Omega \subset X$, the restriction of $\Gamma(u,v)$ to $\Omega$ depends only on the restrictions of $u$ and $v$ to $\Omega$. We write $u \in \calF_{loc}(\Omega)$ if $u \in L^2_{loc}(U,\mu)$ and for every compact subset $K\subset \Omega$ there exists $w \in \calF$ such that $u=w$ $\mu$-a.e. on $K$ (see \cite[p.130]{FOT}). Then, the local character of $\Gamma$ allows to unambiguously define it on $\calF_{loc}(\Omega) \times \calF_{loc}(\Omega)$. The energy measure $\Gamma$ induces a pseudo-metric $\rho$ (called the \emph{intrinsic metric} on $X$) defined for $x,y \in X$ as
\begin{equation}
\rho(x,y):=\sup \{u(x) - u(y) : u \in \calF_{loc}(X) \cap C(X) \text{ and } d\Gamma(u,u)\leq d\mu \text { on } X\},
\end{equation}
where the condition $d\Gamma(u,u) \leq d\mu$ on $X$ means that the measure $\Gamma(u,u)$ is absolutely continuous with respect to $\mu$ and the Radon-Nikodym derivative $d\Gamma(u,u)/d\mu \leq 1$ on $X$. In general $\rho$ could be degenerate in the sense that $\rho(x,y) = \infty$ or $\rho(x,y)=0$ can happen for some $x \ne y$. One way to avoid this degeneracy is to introduce the following:

\medskip
\noindent \textbf{Assumption (A):} All $\rho$-balls $B_\rho(y,R):=\{x \in X : \rho(x,y) < R\}$ are relatively compact in $(X,\tau)$.\\
\noindent \textbf{Assumption (A'):} The topology induced by $\rho$ is equivalent to $\tau$ (the original topology in $X$).
\medskip

Then, under Assumption (A'), it follows that if $X$ is connected and given $x,y \in X$, with $x \ne y$, then $0 < \rho(x,y) < \infty$, thus turning $\rho$ into a metric on $X$ (see \cite[Section 4.2]{St1}). Assumptions (A) and (A') will also be used to guarantee \eqref{Gammarhoyr} below.
\begin{definition}
Let $(X, \tau)$ and $\mu$ be as above and assume that $X$ is connected and Assumptions (A) and (A') hold true. Let $(\calF, \calE)$ be a regular, strongly local Dirichlet form on $L^2(X,\mu)$. 
Following \cite[p.38]{BM3}, let $\Omega \subset X$ be open, and for $u \in \calF_{loc}(\Omega)$ write $\alpha(u,u)=d\Gamma(u,u)/d\mu$, the Radon-Nikodym derivative. Given $1 \leq p < \infty$, the Dirichlet-Sobolev space $D_p[\calE, \Omega]$ is defined as
$$
D_p[\calE, \Omega]:=\{u \in \calF_{loc}(\Omega): \alpha(u,u)\in L^1_{loc}(\Omega, \mu), \ \int_\Omega \alpha(u,u)^{\frac{p}{2}} d\mu + \int_\Omega |u|^p d\mu < \infty \}.
$$
\end{definition}

Given $1\leq p < \infty$ and $1< \sigma <\infty$, we say that the Dirichlet form $(\calF, \calE)$ admits a \emph{weak $(p\sigma, p)$-Sobolev inequality} with a (finite) constant $C_S > 0$ if for every $\rho$-ball $B:=B_\rho(y,R) \subset X$ and every function $\varphi \in D_p[\calE, B]$ with $\text{supp}(\varphi) \subset B$ it holds true that
\begin{align}\label{weakDirichletSob}
\left(\frac{1}{\mu(B)} \int_B |\varphi|^{p \sigma} d\mu \right)^{\frac{1}{p \sigma}} & \leq C_S R \left(\frac{1}{\mu(B)} \int_B \alpha(u,u)^{\frac{p}{2}} d\mu  \right)^{\frac{1}{p}}\\\nonumber
& + C_S \left(\frac{1}{\mu(B)} \int_B |\varphi|^p d\mu\right)^{\frac{1}{p}}.
\end{align}

Then we have the following Dirichlet-form version of Theorem \ref{thm:Sob=>doubling}

\begin{theorem}\label{thm:Sob=>doublingDirichlet} Let $(X, \tau)$ and $\mu$ be as above and assume that $X$ is connected and that Assumptions (A) and (A') hold true. Let $(\calF, \calE)$ be a regular, strongly local Dirichlet form on $L^2(X,\mu)$. Suppose that, for some $1\leq p < \infty$ and $1< \sigma <\infty$, the Dirichlet form $(\calF, \calE)$ admits a weak $(p\sigma, p)$-Sobolev inequality with a constant $C_S >0$. Then, the measure $\mu$ is doubling on $(X, \rho)$. More precisely, there exists a constant $C_D \geq 1$, depending only on $p$, $\sigma$, and $C_S$, such that
\begin{equation*}
\mu(B_\rho(y,2R)) \leq C_D \, \mu(B_\rho(y,R)) \quad \forall y \in X, R > 0.
\end{equation*}
\end{theorem}

\begin{proof} Given a $\rho$-ball $B:=B_\rho(y,R)$, set $B^*:=B_\rho(y,2R)$. By \cite[Lemma 1' on p.191]{St1}, assumptions (A) and (A') imply that for every $y \in X$ and every $r > 0$, the function $\rho_{y,r}: x \mapsto (r- \rho(x,y))^+$ satisfies $\rho_{y,r} \in \calF \cap C_c(X)$ and
\begin{equation}\label{Gammarhoyr}
d\Gamma(\rho_{y,r}, \rho_{y,r}) \leq d\mu.
\end{equation}
Then, for each $j \in \na$ set $r_j:= (2^{-j-1} + 2^{-1}) R$ and, just as in the case for metric spaces in \eqref{defpsijMetric},  define
\[
\psi_j(x):= \left(\frac{r_j - \rho(x,y)}{r_j - r_{j+1}} \right)^+ \wedge 1.
\]
For each $j \in \na$ define $s_j:=1/ (r_j - r_{j-1}) = 2^{j+2}/R$ and the set  $D_j:= \{x \in X: s_j (r_j - \rho(x,y)) > 1\}$ so that, by the definition of $\psi_j$, the so-called truncation property (see \cite[p.190]{St1}), the fact that $d\Gamma(1,1) =0$ (due to the locality of $\Gamma$, see \cite[p.189]{St1}), and the estimate \eqref{Gammarhoyr} we have
\begin{align*}
d\Gamma(\psi_j, \psi_j) & = d\Gamma(s_j \rho_{y, r_j} \wedge 1,s_j \rho_{y, r_j} \wedge 1) = 1_{D_j} d\Gamma(1,1) + 1_{X \setminus D_j} d\Gamma(s_j \rho_{y, r_j}, s_j \rho_{y, r_j}) \\
& = 1_{X \setminus D_j} d\Gamma(s_j \rho_{y, r_j}, s_j \rho_{y, r_j}) = 1_{X \setminus D_j} s_j^2 d\Gamma(\rho_{y, r_j}, \rho_{y, r_j}) \leq s_j^2 d\mu.
\end{align*}
Hence,
\begin{equation}\label{gradientboundDirichlet}
d\Gamma(\psi_j, \psi_j) \leq \left(\frac{2^{j+2}}{R}\right)^2 d\mu\quad \forall j \in \na.
\end{equation}
As usual, define $B_j$ by
$$
\frac{1}{2}B \subset B_j:= \{x \in X: \rho(x,y) \leq r_j\} \subset B \subset B^*=2B,
$$
and the proof of Theorem \ref{thm:Sob=>doublingDirichlet} follows along the same lines as the one for Theorem \ref{thm:Sob=>doublingMetric}. Notice that, alternatively, one could use the construction from \cite[Section 3]{BM2} to produce $\{\psi_j\}_{j \in \na} \subset \calF_{loc}(X) \cap C(X)$ with $\psi_j \equiv 1$ on $B_{j+1}$, $\psi_j \equiv 0$ on $X\setminus B_j$, $0 \leq \psi_j \leq 1$, and
$$
d\Gamma(\psi_j, \psi_j) \leq 10\left(\frac{2^{j+2}}{R}\right)^2 d\mu \quad \forall j \in \na.
$$

\end{proof}

\subsection{A subelliptic version} In $(\rn,d_E)$ as well as in the metric-space and Dirichlet-form contexts above, the various sequences $\{\psi_j\}_{j \in \na}$ had ``bounded gradients'' in the sense of \eqref{condpsij} and \eqref{gradientboundDirichlet}. These uniform bounds were used to obtain the corresponding inequalities of the type \eqref{boundGradPsij}. Next we will show how in certain subelliptic contexts the uniform bounds on the gradients can be weakened to suitable integral bounds. This will allow us to relate our Theorem \ref{thm:Sob=>doubling} to the notion of accumulating sequence of Lipschitz cut-off functions (see Remark \ref{ASOLF}).

We largely follow the terminology from \cite[Section 1]{SW1}. Consider an open subset $\Omega\subset\mathbb{R}^n$ (in the Euclidean topology) and let
$$
Q : \Omega \rightarrow \{\text{non-negative semi-definite }  n \times n \text{ matrices}\}
$$
be a locally bounded function on $\Omega$. For a Lipschitz function $u : \Omega \rightarrow \re$ (throughout this subsection, Lipschitz means Lipschitz with respect to the Euclidean distance), define its $Q$-gradient Lebesgue-a.e. in $\Omega$ as
$$
[\nabla u]_Q:= ( \nabla u^T Q \nabla u)^{\frac{1}{2}}.
$$

\begin{definition}
Let $d : \Omega \times \Omega \rightarrow [0, \infty)$ be a metric on $\Omega$ that generates a topology equivalent to the Euclidean topology in $\Omega$. 
Given $s >1$ we say that the structure $(\Omega, d, Q)$ admits \emph{accumulating sequences of Lipschitz cut-off functions} with exponent $s$ (see \cite[p.9]{SW1}) if there exist constants $0 < \nu < 1$, $K > 0$, and $N > 1$ such that for every $d$-ball $B:=B_d(y,R)$, with $R < \mathrm{dist}(y, \partial \Omega) /6$,  there exists a sequence $\{\psi_j\}_{j \in \na}$ of functions defined on $\Omega$ such that
\begin{align}\label{suppPsi1inB}
\mathrm{supp}(\psi_1) &\subset B,\\ \label{BnuRinPsij=1}
 B_d(y,\nu R) &\subset \{x:\psi_j(x)=1\} \: \forall j \in \na,\\ \label{suppPsij1}
 \mathrm{supp} (\psi_{j+1}) &\subset \{x:\psi_j(x)=1\}\: \forall j \in \na,\\ \label{psijLips01}
\psi_j \text{ is Lipschitz and }  0 & \le \psi_j\le 1\: \forall j \in \na,\\
\left(\frac{1}{|B|}\int_{B}[\nabla\psi_j]_Q^s \, dx \right)^{\frac{1}{s}} &\leq \frac{K  N^j}{R}\: \forall j \in \na,\label{s-averageGradQ}
\end{align}
where $|B|$ stands for the Lebesgue measure of $B$. 
\end{definition}

For $1 \leq p < \infty$, let $\mathcal{W}^{1,p}_Q(\Omega, dx)$ denote the closure of the Lipschitz functions on $\Omega$ under the norm
$$
\norm{u}{\mathcal{W}^{1,p}_Q(\Omega, dx)}:= \norm{u}{L^p(\Omega, dx)} + \norm{[\nabla u]_Q}{L^p(\Omega,dx)}.
$$
Now for $1 \leq p < \infty$ and $1 < \sigma < \infty$ we say that  $(\Omega, d, Q)$ admits a \emph{weak $(p\sigma, p)$-Sobolev inequality} with a (finite) constant $C_S > 0$ if for every $d$-ball $B:=B_d(y,R) \subset X$, with $0 < R < \text{dist}(y, \partial \Omega)/6$, and every function $\varphi \in \mathcal{W}^{1,p}_Q(\Omega, dx)$ with $\text{supp}(\varphi) \subset B$ it holds true that
\begin{align}\label{weakSubellipticSob}
\left(\frac{1}{|B|} \int_B |\varphi|^{p \sigma} \, dx \right)^{\frac{1}{p \sigma}} & \leq C_S R \left(\frac{1}{|B|} \int_B [\nabla \varphi]_Q^p \, dx \right)^{\frac{1}{p}}\\\nonumber
& + C_S \left(\frac{1}{|B|} \int_B |\varphi|^p \, dx \right)^{\frac{1}{p}}.
\end{align}

As usual, given $r > 1$ denote its dual H\"older exponent as $r'$, defined by $r'+r=rr'$. Then we have the following  formulation of Theorem \ref{thm:Sob=>doubling}

\begin{theorem}\label{thm:Sob=>doublingSubelliptic} Suppose that, for some $1\leq p < \infty$, $1< \sigma <\infty$, and $s > p \sigma'$, the structure $(\Omega, d, Q)$ admits accumulating sequences of Lipschitz cut-off functions with exponent $s$ as well as a weak $(p\sigma, p)$-Sobolev inequality with a constant $C_S >0$. Then, Lebesgue measure $dx$ is doubling on $(\Omega, d)$. More precisely, there exists a constant $C_D \geq 1$, depending only on $p$, $\sigma$, $s$, $K$, $N$, and $C_S$, such that
\begin{equation*}
|B_d(y,2R)| \leq C_D \, |B_d(y,R)| \quad \forall y \in X, \forall \: 0 < R < \text{dist}(y, \partial \Omega)/6.
\end{equation*}
\end{theorem}

\begin{proof} Given a $d$-ball $B:=B_d(y,R)$, with  $0 < R < \text{dist}(y, \partial \Omega)/6$, just as in the proof of Theorem \ref{thm:Sob=>doublingMetric}, apply the weak-Sobolev inequality \eqref{weakSubellipticSob} to the accumulating sequence of Lipschitz cut-off functions $\{\psi_j\}$ on the $d$-ball $B^*:=B_d(y,2R)$ and, for $j \in \na$, set $B_j:=\mathrm{supp}(\psi_j)$, to obtain
\begin{align*}
\left(\frac{1}{|B^*|} \int_{B_j} |\psi_j|^{p \sigma} \, dx \right)^{\frac{1}{p \sigma}} & \leq 2 C_S R \left(\frac{1}{|B^*|} \int_{B_j} [\nabla \psi_j]_Q^p \, dx \right)^{\frac{1}{p}}\\\nonumber
& + C_S \left(\frac{1}{|B^*|} \int_{B_j} |\psi_j|^p \, dx\right)^{\frac{1}{p}}.
\end{align*}
The main step in the proof is to find a substitute for \eqref{boundGradPsij}. Set $q:=s/p > 1$ so that $q':=s/(s-p) >1$. By applying H\"older's inequality with $q$ and $q'$ and using \eqref{s-averageGradQ}, we get
\begin{align*}
\left(\frac{1}{|B^*|} \int_{B_j} [\nabla \psi_j]_Q^p dx \right)^{\frac{1}{p}} &\leq  \left(\frac{|B_j|}{|B^*|} \right)^{\frac{1}{p} - \frac{1}{s}} \left(\frac{1}{|B^*|} \int_{B_j} [\nabla \psi_j]_Q^s dx \right)^{\frac{1}{s}} \\
& \leq  \frac{KN^j}{2R} \left(\frac{|B_j|}{|B^*|} \right)^{\frac{1}{p} - \frac{1}{s}}.
\end{align*}
By properties \eqref{suppPsij1} and \eqref{psijLips01}, and from the inequalities above, it follows that
\begin{align*}
 \left(\frac{|B_{j+1}|}{|B^*|} \right)^{\frac{1}{p \sigma}}& \leq \left(\frac{1}{|B^*|} \int_{B_j} |\psi_j|^{p \sigma} dx \right)^{\frac{1}{p \sigma}} \leq KN^j \left(\frac{|B_j|}{|B^*|} \right)^{\frac{1}{p} - \frac{1}{s}} + \left(\frac{|B_j|}{|B^*|} \right)^{\frac{1}{p}}\\
 & \leq  (KN^j + 1) \left(\frac{|B_j|}{|B^*|} \right)^{\frac{1}{p} - \frac{1}{s}} \leq  (K+1)N^{j} \left(\frac{|B_j|}{|B^*|} \right)^{\frac{1}{p} - \frac{1}{s}}.
\end{align*}
Now, set $\beta:= \sigma/q' = \sigma (1-p/s)$ and notice that the hypothesis $s > p \sigma'$ means $\beta > 1$. Raising the inequality above to the power $ps/(s-p) >0$ yields
\begin{equation}\label{Ej1beta}
 \left(\frac{|B_{j+1}|}{|B^*|} \right)^{\frac{1}{\beta}} \leq (K+1)^{\frac{p \sigma}{\beta}} N^\frac{j p \sigma}{\beta} \left(\frac{|B_j|}{|B^*|} \right).
\end{equation}
Set $P_j:= |B_j|^{\frac{1}{\beta^{j-1}}}$ so that raising \eqref{Ej1beta} to the power $1/\beta^{j-1}$ implies
$$
P_{j+1} \leq   (K+1)^{\frac{p \sigma}{\beta^j}} N^\frac{j p \sigma}{\beta^j} |B^*|^{\frac{1-\beta}{\beta^j}} P_j
$$
From conditions \eqref{suppPsi1inB} and \eqref{BnuRinPsij=1} we get $\lim\limits_{j \to \infty} P_j =1$ and $P_1 \leq |B|$ and the proof can now be completed just as in Section \ref{sec:proofThm}. \end{proof}

\begin{remark}\label{ASOLF} Notice that the family $\{\psi_j\}_{j \in \na}$ in \eqref{defpsijMetric} (with respect to Euclidean distance) will satisfy \eqref{suppPsi1inB} through \eqref{psijLips01} (with $\nu=1/2$) as well as \eqref{s-averageGradQ} with $\norm{Q}{L^\infty(B)}$ instead of $K$. The fact that the bound $K$ must be uniform in $B$ leads to the study of the interaction between the Euclidean distance and the distance $d$ whose balls define the inequality  \eqref{s-averageGradQ} and the Sobolev inequality \eqref{weakSubellipticSob}. Sufficient conditions on the interaction between Euclidean balls and $d$-balls for the structure $(\Omega, d, Q)$ to admit accumulating sequences of Lipschitz cut-off functions have been found in \cite[Proposition 68]{SW1} and \cite[Lemma 8]{KR}. The hypotheses on the validity of Sobolev inequalities, the existence of accumulating sequences of Lipschitz cut-off functions, and doubling property for Lebesgue measure on the $d$-balls, are typical in the related literature (see, for instance, \cite{MRW, SW1, SW2}). Theorem \ref{thm:Sob=>doublingSubelliptic} now renders the doubling condition redundant.
\end{remark}

\section*{Acknowledgements}

The authors would like to thank the anonymous referees whose suggestions helped improve the presentation of the article.

\end{document}